\documentclass[a4paper, 12pt, reqno]{amsart}
\usepackage{amsmath,amsthm,amssymb,framed,verbatim}
\usepackage[hidelinks]{hyperref}
\usepackage{xcolor}
\usepackage{graphicx}
\usepackage[a4paper,  margin=2.5cm]{geometry}

\newcommand{\norm}[1]{\Vert #1 \Vert}  
\newcommand{\abs}[1]{\vert #1 \vert}  
\newcommand{\ov}{\overline}

\DeclareMathOperator{\tr}{trace}

\DeclareMathOperator{\Hess}{Hess}

\newcommand{\dif}{\mathrm{d}}
\newcommand{\ii}{\mathrm{i}}
\newcommand{\CC}{\mathbb{C}}

\newcommand{\HH}{\mathcal{H}}

\newcommand{\RR}{\mathbb{R}}   
 
\newcommand{\NN}{\mathbb{N}}
\newcommand{\Ss}{\mathbb{S}}
\newcommand{\Jj}{\mathcal{J}} 
\newcommand{\Aa}{\mathcal{A}}

\newtheorem{te}{Theorem}
\newtheorem{pr}{Proposition}

\newtheorem{lm}{Lemma}
\theoremstyle{definition}
\newtheorem{de}{Definition} 

\newtheorem{ex}{Example}

\begin{document}

\title[Index of conformal Hopf maps]{The index of conformal Hopf maps}

\author{E. Loubeau}
\address{University of Brest, LMBA UMR 6205, F-29870 Brest, France}
\email{loubeau@univ-brest.fr}

\author{R. Slobodeanu}
\address{University of Bucharest, Faculty of Physics, P.O. Box Mg-11, RO--077125 Bucharest-M\u agurele, Romania}
\email{radualexandru.slobodeanu@g.unibuc.ro}

\subjclass[2010]{58E20, 53C43, 58E30, 53B50.}

\date{\today}

\keywords{Harmonic map, calculus of variations, stability.}

\begin{abstract}
\noindent
We compute the index and nullity of any harmonic morphism from $\Ss^3$ to $\Ss^2$.
\end{abstract}

\maketitle

\section{\textcolor{blue}{\textbf{Introduction}}}
Given two Riemannian manifolds $(M, g)$ and $(N,h)$, the \textit{Dirichlet energy} of a (smooth) map $\varphi: M \to N$ is
\begin{equation*}
E(\varphi) =
\frac{1}{2} \int_M \abs{\dif \varphi}^2 \upsilon_{g}.
\end{equation*}
If for any (smooth, compactly supported) 1-parameter variation $\{\varphi_t\}$ of $\varphi=\varphi_0$ we have $\frac{d}{d t} \big\lvert _{t=0}E(\varphi_t)=0$, then $\varphi$ is called a \textit{harmonic map} \cite{eell} and is characterized by the equation $\tr \nabla \dif \varphi=0$.  At such critical point, if $M$ is compact, we  have \cite{smi}
$$
\frac{d^2}{d t^2} \Bigl\lvert _{t=0}E(\varphi_t)=
\Hess_{\varphi}(v,v)=\int_M \left \langle \mathcal{J}_{\varphi}(v), v \right \rangle \upsilon_g,
$$ 
where   $v(p) = \frac{\partial \varphi_t}{\partial t}(p) \big\lvert _{t=0} \in T_{\varphi(p)}N$ at any $p\in M$ and
$$\mathcal{J}_{\varphi}:\Gamma(\varphi^{-1}TN) \to \Gamma(\varphi^{-1}TN), \quad
 \mathcal{J}_{\varphi}(v)=-\tr (\nabla^\varphi)^2 v - \tr R^N(v,\dif\varphi)\dif\varphi$$ 
is a linear self-adjoint elliptic operator, called the Jacobi operator of $\varphi$. Here $\nabla^\varphi$ denotes the pull-back connection in the pull-back bundle $\varphi^{-1}TN$.

The number $\dim (\ker \mathcal{J}_{\varphi})=\mathrm{nullity}(\varphi)$ is the \textit{nullity} of $\varphi$ and the sum of the dimensions
of the eigenspaces of $\mathcal{J}_{\varphi}$ corresponding to the negative eigenvalues is the (Morse) \textit{index}, denoted as $\mathrm{index}(\varphi)$.

If $\frac{d^2}{d t^2} \big\lvert _{t=0}E(\varphi_t)\geqslant 0$, i.e. $\mathrm{index}(\varphi)=0$, then 
$\varphi$ is a \textit{weakly stable} harmonic map. By a standard argument of Xin \cite{xin}, if $m > 2$, then there is no non-constant stable harmonic map defined on the unit round sphere $\Ss^m$. In particular, the \textit{Hopf map} $\phi: \Ss^3 \to \Ss^2$ is unstable and its index was computed to be equal to $4$ \cite{loub, ura}. This is due to the four non-Killing conformal fields of $\Ss^3$ (gradients of $\lambda_1$-eigenfunctions) along which the energy decreases, as shown more generally in  \cite{soufi}: 

\begin{center}
$\mathrm{index}(\varphi)\geqslant m+1$ for any $\varphi: \Ss^m \to N$ and $m\geqslant 3$.    
\end{center}

\noindent A \textit{low-index harmonic map} on the round sphere is a harmonic map for which equality holds \cite{kalt}. The Hopf map $\Ss^{2n+1} \to \CC P^n$ and $Id_{\Ss^m}$ are  low-index harmonic maps  \cite{chen, lar, smi, sur}.

For nullity, since any harmonic map $\varphi$ satisfies \cite{ohn}
\begin{equation} \label{ohnita}
\mathcal{J}_{\varphi}(\dif \varphi(X))=
-\dif \varphi(\tr \nabla^2 X + \mathrm{Ric}^M (X)) -2\tr \nabla \dif \varphi(\cdot, \nabla_{\cdot} X), \quad X\in \Gamma(TM),
\end{equation}
it is clear that $\mathcal{J}_{\varphi}(\dif \varphi(X))=0$ if $X$ is Killing and this yields a lower bound of nullity in terms of the dimension of the isometry group of $M$.  For the Hopf map on $\Ss^3$, post-composition with holomorphic diffeomorphisms (i.e. conformal, but not necessarily isometric) give new harmonic maps of the same energy level. Thus holomorphic vector fields on the target (which are not Killing) yield new genuine Jacobi fields. The full automorphism group of the target is the 6-dimensional M\"obius group $PSL(2,\mathbb{C})$ acting on $\mathbb{CP}^1$. Of its 6 generators, 3 are the rotations (Killing, but already counted in as part of the domain-symmetries) and 3 are the non-compact (hyperbolic) conformal directions gradients of  eigenfunctions of the Laplacian. These provide new Jacobi fields. In total, for the Hopf map $\phi$, we have Jacobi fields that are pushforward of the $5 = 6 - 1$ Killing fields of the domain (subtracting the one in the kernel of $d\phi$) and $3$ extra Jacobi fields from the M\"obius group. Hence $\mathrm{nullity}(\phi)=8$ as it has obtained in \cite{loub, ura}.

In this paper we extend these well-known results on the index and nulity of the Hopf map $\phi$ to any harmonic morphism $\varphi:\Ss^3 \to \Ss^2(1/2) \simeq  \CC P^1$. These maps are known to be \cite{BWclassif} the composition of an isometry of $\Ss^3$, the Hopf fibration and a holomorphic map from $\CC P^1$ into itself. For more information about harmonic morphisms we refer to \cite{BW}. Our main result is

\begin{te}
Let $\psi:\CC P^1 \to \CC P^1$ be a holomorphic map of degree $d$ and $\phi: \Ss^3 \to \CC P^1$ be the Hopf map. Then $\psi \circ \phi$ is harmonic and 
\begin{equation}
\mathrm{index}(\psi \circ \phi)=4d, \qquad \mathrm{nullity}(\psi \circ \phi)=8d.   
\end{equation}
 
\end{te}

In \cite{trist} Rivi\`ere proved that a low-index harmonic map from $\Ss^3$ into $\Ss^2$ must be a harmonic morphism, that is, up to an isometry of the domain, such a composition of a holomorphic map of $\CC P^1$ and the Hopf fibration. One consequence of the above result, is that the holomorphic map in Rivi\`ere's result must have degree 1; this fact has been also proved in the paper \cite{chen} that appeared during the last stage of the preparation of our manuscript. Our results have been developed independently and have a stronger conclusion giving the exact value of the index and nullity of  $\psi \circ \phi$.   

\section{\textcolor{blue}{\textbf{Preliminaries}}}
The sphere $\Ss^3$ is the set of points $(z_1, z_2) \in \CC^2$ with $\abs{z_1}^2 + \abs{z_2}^2 =1$ endowed with the canonical metric. Denoting $z_j= x_j + \ii y_j$, $j=1,2$, at each point we have the orthonormal frame (of Killing vector fields):
\begin{equation} \label{stdframe}
\begin{split}
\xi     & = - y_1 \partial_{x_1} + x_1 \partial_{y_1} - y_2 \partial_{x_2} + x_2 \partial_{y_2} ,\\
X & =-x_2 \partial_{x_1} + y_2 \partial_{y_1} + x_1 \partial_{x_2} - y_1 \partial_{y_2} ,\\
Y & =-y_2 \partial_{x_1} - x_2 \partial_{y_1} + y_1 \partial_{x_2} + x_1 \partial_{y_2}.
\end{split}
\end{equation}
satisfying the $\mathfrak{su}(2)$ commutation relations:
\begin{equation} \label{stdframebrackets}
\begin{split}
[\xi,X] = -2Y,\qquad 
[\xi,Y] =  2X,\qquad
[X,Y] = -2\xi.
\end{split}
\end{equation}
We note that $\xi$ is the Hopf field tangent to the fibres of the Hopf fibration $\phi:\Ss^3 \to \Ss^2(1/2) \simeq  \CC P^1$, $\phi(z_1,z_2)=\tfrac{1}{2}(\abs{z_1}^2 - \abs{z_2}^2, 2\ov{z}_1 z_2)$,  while $X$ and $Y$ span the horizontal distribution $(\ker d\phi)^\perp$.

In this paper, we consider the Hopf map postcomposed with a holomorphic map.
Recall that holomorphic maps $\psi : \CC P^1 \to \CC P^1$ are rational functions and we can write $\psi(z)=P(z)/Q(z)$ with relatively prime polynomials $P,Q\in\CC[z]$ so that 
$\deg\psi=\max\{\deg P,\deg Q\}$. The Hopf invariant of $\varphi=\psi \circ \phi$ is $\mathcal{Q}(\varphi)=(\deg \psi)^2$.

\begin{de}
We call (horizontally) \emph{$d$-conformal Hopf map} a map of the form  $\varphi := \psi \circ \phi$ where $\psi$ is a holomorphic map of $\Ss^2(1/2) \simeq  \CC P^1$ of degree $d$ and $\phi$ is the Hopf map. When the degree $d$ is not relevant we say simply \textit{conformal Hopf map}.
\end{de}

Since any (anti-)holomorphic map between Riemann surfaces is a weakly conformal
harmonic map \cite{BW},  a conformal Hopf map is by construction harmonic and horizontally weakly conformal, with basic dilation 
$\lambda=\tilde\lambda\circ\phi$, where $\tilde\lambda$ is the conformal factor of $\psi$. Denote by $J$ the complex structure on $\Sigma:=\mathbb{C}P^1$ so that the  above horizontal frame satisfies
\begin{equation}\label{eq:varphi-complex}
d\varphi(Y)=Jd\varphi(X).
\end{equation}
We have $|d\varphi(X)|=|d\varphi(Y)|=\lambda$.

\begin{ex} 
In spherical coordinates $(u,v)$, the holomorphic map $\psi(z) = z^d$, $d\in \NN^*$, viewed as a map from the 2-sphere to itself, takes the form:
$$
(\cos u, \sin u e^{iv}) \mapsto (\cos \tilde{u}, \sin \tilde{u} e^{\ii dv}), \quad
\tilde{u} = 2 \arctan \left(\tan \frac{u}{2}\right)^d\,.
$$
For this map we can compute:
\begin{equation}
\lambda^2 = \frac{d^2 \sin^{2 d - 2}s \, \cos^{2 d - 2}s}{\left(\sin ^{2 d}s + \cos ^{2 d}s\right)^2}, 
\end{equation}
where $\Ss^3$ is parametrized as $(\cos s \, e^{i\phi_1}, \sin s \, e^{i\phi_2})$, $s \in [0, \frac{\pi}{2}]$, $\phi_i \in [0, 2\pi)$.
\end{ex}

The goal of this paper is to compute the index and nullity of this class of mappings. Recall that the Hessian of a conformal Hopf map is given by \cite[(4.8.2)]{BW}
\begin{equation}\label{eq:hessian}
\Hess_{\varphi}(v,v)= \int_{\Ss^3}
\left(
|\nabla^\varphi_\xi v|^2+
|\nabla^\varphi_X v|^2+
|\nabla^\varphi_Y v|^2-
4\lambda^2|v|^2
\right)\,\upsilon_g,
\end{equation}
and the Jacobi operator by
\begin{equation}\label{eq:jacs3}
\mathcal{J}_{\varphi}(v)=-(\nabla^\varphi)_{\xi, \xi}^2 v-(\nabla^\varphi)_{X,X}^2 v-(\nabla^\varphi)_{Y,Y}^2 v - 4\lambda^2v.
\end{equation}
Here $\varphi^{-1}T\Sigma$ is a real rank-two bundle, which we see as a complex line bundle by putting $\ii v:=Jv$. Thus, if the variation fields remain real vector fields along $\varphi$, their space carries the natural complex structure induced by $J$. When identifying 
$(\varphi^{-1}T\Sigma,J)\simeq \varphi^{-1}T^{1,0}\Sigma$, $V$ will denote the corresponding section of the $(1,0)$-bundle.

Let $\mathcal{H}$ be the complex Hilbert space of $L^2$-sections of $\varphi^{-1}T\Sigma$
and $R_t(p):= pe^{\ii t}$, $t\in\mathbb R/2\pi\mathbb Z$ denote the right $\Ss^1$ action on $\Ss^3$.
Since $\varphi\circ R_t=\varphi$, $v(R_t p)$ and $v(p)$ belong to the same fibre $T_{\varphi(p)}\Sigma$. Hence
\begin{equation}\label{eq:Ut}
(U_t v)(p):=v(R_t p)
\end{equation}
defines a representation $U:\Ss^1\to\mathrm{U}(\mathcal H)$, $U(e^{it}):= U_t$, which is unitary and strongly continuous. The first property is due to the fact that  $R_t$ is an isometry of $\Ss^3$ and
therefore preserves the volume measure, while the latter is immediate for smooth sections, and for
$v\in\mathcal H$ it follows by approximating $v$ in $L^2$ by smooth
sections and using that  $U_t$ is unitary.

For $m\in\mathbb Z$, define the $m$-th weight space by 
\begin{equation}\label{eq:Hm}
\mathcal H_m
:=
\{v\in\mathcal H:U_t v=e^{-\ii mt}v\text{ for every }t\}\, ,
\end{equation}
or, equivalently $\mathcal H_m = \{v\in\mathcal H:v(pe^{\ii t})=\cos(mt)v(p)-\sin(mt)Jv(p), \forall t\}$, and use the decomposition theorem for unitary representations of compact groups (see, e.g., \cite[\S3.1 and \S5.1]{folland}) to state that
\begin{equation}\label{desc}
\mathcal H = \widehat{\bigoplus}_{m\in\mathbb Z}\mathcal H_m.
\end{equation}
Recall that $\Gamma(\varphi^{-1}TN)$ denotes the space of \textit{smooth} sections in the pull-back and set
$$
\mathcal H_m^\infty:=\mathcal H_m\cap \Gamma(\varphi^{-1}TN).
$$
We have the characterization (cf. \cite{pss} for a similar construction):
\begin{equation}\label{eq:Hm-smooth}
v\in\mathcal H_m^\infty
\quad\Longleftrightarrow\quad
\nabla^\varphi_\xi v=-mJv.
\end{equation}
due to the fact that,  along the flow of  $\xi$, $\nabla^\varphi_\xi v=-mJv$ translates into a simple ODE for the map $t\mapsto v(pe^{\ii t})\in T_{\varphi(p)}\Sigma$:
\[
\frac{d}{dt}v(pe^{\ii t})=-mJv(pe^{\ii t}).
\]

Define the first-order operator
\begin{equation}
D: \Gamma(\varphi^{-1}T\Sigma) \to \Gamma(\varphi^{-1}T\Sigma), \quad
D=\nabla^\varphi_X+J\nabla^\varphi_Y.
\label{eq:D}
\end{equation}
As $J$ is $\nabla^\Sigma$-parallel and $X,Y$ are divergence free, the $L^2$-adjoint of $D$ is given by 
\begin{equation}
D^*=-\nabla^\varphi_X+J\nabla^\varphi_Y.
\end{equation}

\begin{pr}\label{prop-prelim}
Let $\varphi$ be a $d$-conformal Hopf map.

\noindent $(i)$ The Jacobi operator commutes with the formally self-adjoint operator $J\nabla^\varphi_\xi$:
\begin{equation} \label{Jcommute}
[\mathcal{J}_{\varphi}, J\nabla^\varphi_\xi]=0.
\end{equation}
In particular, each $\mathcal H_m^\infty$ is invariant under $\mathcal J_\varphi$ (and so is the corresponding $L^2$-isotypical subspace for the self-adjoint realization of $\mathcal J_\varphi$).

\noindent $(ii)$ If $v \in \mathcal{H}_m$ is smooth, then 
\begin{equation}\label{HessHm}
\Hess_{\varphi}(v,v)= \int_{\Ss^3}
\left(
|D v|^2
+
m(m-2)|v|^2
\right)\,\upsilon_g,
\end{equation}
and, on $\mathcal{H}_m^\infty$, the Jacobi operator reads
\begin{equation}\label{eq:Jacobi-mode}
\mathcal J_\varphi
=D^*D+m(m-2)\operatorname{Id}.
\end{equation}

\noindent $(iii)$ If $v \in \mathcal{H}_m$ is smooth, then $Dv \in \HH_{m-2}^\infty$.
\end{pr}

\begin{proof}
\noindent $(i)$ Set
\[
A:=\nabla_\xi^\varphi,\qquad
B:=\nabla_X^\varphi,\qquad
C:=\nabla_Y^\varphi.
\]
Since, for the chosen orthonormal frame,
$\nabla_\xi^{S^3}\xi=
\nabla_X^{S^3}X=
\nabla_Y^{S^3}Y=0$,
we have, cf. \eqref{eq:jacs3},
\[
\Jj_\varphi
=
-(A^2+B^2+C^2)-4\lambda^2 \mathrm{Id}.
\]

For the pullback connection one has
\[
[\nabla_U^\varphi,\nabla_V^\varphi]
=
\nabla_{[U,V]}^\varphi
+
R^\Sigma\bigl(d\varphi(U),d\varphi(V)\bigr).
\]
Since $d\varphi(\xi)=0$, the curvature term vanishes whenever one of the entries is $\xi$. Hence,
using \eqref{stdframebrackets}
we obtain
\begin{equation} \label{com_nabla_phi}
[\nabla_\xi^\varphi,\nabla_X^\varphi]
=
\nabla_{[\xi,X]}^\varphi
=
-2\nabla_Y^\varphi,
\qquad 
[\nabla_\xi^\varphi,\nabla_Y^\varphi]
=
\nabla_{[\xi,Y]}^\varphi
=
2\nabla_X^\varphi.
\end{equation}
Thus, with the notations above,
we have
\[
[A,B]=-2C,
\qquad
[A,C]=2B.
\]
Hence
\[
[A,B^2+C^2]
=
[A,B]B+B[A,B]+[A,C]C+C[A,C]
=0.
\]
and, obviously, $[A,A^2]=0$.

Since $\lambda$ is constant along the Hopf fibres, $\xi(\lambda)=0$,
and hence
\[
[A,\lambda^2 \mathrm{Id}]=0.
\]
Gathering all commutations proved above, we get  $[A,\mathcal J_\varphi]=0$. Since $J$ is parallel and the curvature term in \eqref{eq:jacs3} is the scalar endomorphism $4\lambda^2\operatorname{Id}$, we obtain \eqref{Jcommute}.

If $v\in\mathcal H_m^\infty$, then $Av=-mJv$, and therefore
\[
A(\mathcal J_\varphi v)
=\mathcal J_\varphi(Av)
=-m\mathcal J_\varphi(Jv)
=-mJ\mathcal J_\varphi v.
\]
Hence $\mathcal J_\varphi v\in\mathcal H_m^\infty$.

\medskip
\noindent \textit{(ii)} Put
\[
I:=\int_{\Ss^3}\langle\nabla_X^\varphi v,J\nabla_Y^\varphi v\rangle\,\upsilon_g.
\]
Since $X$ and $Y$ are divergence free and $J$ is parallel, integration by parts in the $X$- and $Y$-directions gives
\[
2I=\int_{\Ss^3}\left\langle
v,J(\nabla_Y^\varphi\nabla_X^\varphi-\nabla_X^\varphi\nabla_Y^\varphi)v
\right\rangle\,\upsilon_g.
\]
For the pullback connection,
\[
\nabla_Y^\varphi\nabla_X^\varphi-\nabla_X^\varphi\nabla_Y^\varphi
=\nabla_{[Y,X]}^\varphi+R^\Sigma(d\varphi(Y),d\varphi(X)).
\]
Using $[Y,X]=2\xi$ and $\nabla_\xi^\varphi v=-mJv$, the bracket term contributes $2m|v|^2$. Since $\Sigma=\Ss^2(1/2)$ has sectional curvature $4$, while $d\varphi(Y)=Jd\varphi(X)$ and $|d\varphi(X)|=\lambda$, we also have
\[
J R^\Sigma(d\varphi(Y),d\varphi(X))v=-4\lambda^2v.
\]
Consequently
\begin{equation}\label{twoI}
2I=\int_{\Ss^3}(2m-4\lambda^2)|v|^2\,\upsilon_g,
\end{equation}
and therefore
\[
\int_{\Ss^3}|Dv|^2\,\upsilon_g
=
\int_{\Ss^3}
\bigl(|\nabla_X^\varphi v|^2+|\nabla_Y^\varphi v|^2+(2m-4\lambda^2)|v|^2\bigr)\,\upsilon_g.
\]
Since $|\nabla_\xi^\varphi v|^2=m^2|v|^2$, comparison with \eqref{eq:hessian} yields
the desired identity \eqref{HessHm}.

Finally, notice that
\[
D^*D=-(B^2+C^2)+J(CB-BC),
\]
and the computation above gives
\[
J(CB-BC)v=(2m-4\lambda^2)v,
\]
which, together with $A^2v=-m^2v$, yields \eqref{eq:Jacobi-mode} from \eqref{eq:jacs3}.

\medskip
\noindent \textit{(iii)} From \eqref{com_nabla_phi} and $\nabla^\varphi J=0$,
\[
[\nabla_\xi^\varphi,D]
=-2\nabla_Y^\varphi+2J\nabla_X^\varphi
=2JD.
\]
Since $D(Jv)=JDv$, for $v\in\mathcal H_m^\infty$,
\[
\nabla_\xi^\varphi(Dv)
=D(\nabla_\xi^\varphi v)+2JDv
=-mJDv+2JDv
=-(m-2)JDv.
\]
Thus $Dv\in\mathcal H_{m-2}^\infty$.
\end{proof}

\section{\textcolor{blue}{\textbf{Proof of the main theorem}}}

According to Proposition \ref{prop-prelim} and the orthogonal decomposition
\eqref{desc}, the index and nullity of a conformal Hopf map $\varphi$ can be computed from the restrictions of the Jacobi operator to the spaces $\mathcal H_m$.

The following technical lemma gives a coercivity estimate that will be needed later.

\begin{lm}
The following inequalities hold true 
\begin{align}
\norm{D^* w}_{L^2}^2 & \geqslant 2\norm{w}_{L^2}^2,
\quad \forall w \in \HH_{-1}^\infty, \label{DDstar}\\ 
\norm{D v}_{L^2}^2
&\geqslant 2\norm{v}_{L^2}^2,
\quad \forall v \in \HH_{1}^\infty \cap (\ker D)^\perp. \label{DstarD}
\end{align}
\end{lm}

\begin{proof}
Let $w\in\mathcal H_m^\infty$. Expanding the square and using
\eqref{twoI}, we have
\begin{align*}
\|D^*w\|_{L^2}^2
&=
\|\nabla_X^\varphi w\|_{L^2}^2
+
\|\nabla_Y^\varphi w\|_{L^2}^2
-2\int_{\Ss^3}
\langle\nabla_X^\varphi w,J\nabla_Y^\varphi w\rangle\,\upsilon_g
\\
&=
\|\nabla_X^\varphi w\|_{L^2}^2
+
\|\nabla_Y^\varphi w\|_{L^2}^2
+
\int_{\Ss^3}(4\lambda^2-2m)|w|^2\,\upsilon_g.
\end{align*}
For $m=-1$ this gives \eqref{DDstar} that extends from smooth sections to the whole domain
by density:
\begin{equation}\label{eq:Dstar-gap}
\|D^*w\|_{L^2}^2\ge 2\|w\|_{L^2}^2,
\qquad w\in\mathcal H_{-1}.
\end{equation}
Here we consider
\[
D:\mathcal D(D)\subset\mathcal H_1\longrightarrow\mathcal H_{-1}
\]
be the closed $L^2$-realization of the first-order operator $D$ and $D^*$ to be its adjoint.

We now deduce the second estimate. Notice that even if $D^*D$ has discrete spectrum by \eqref{eq:Jacobi-mode}   we do not know a priori if same is true for $DD^*$ (if this were already known, the conclusion would follow immediately from the fact that \(DD^*\) and \(D^*D\) have the same nonzero eigenvalues, counted with multiplicity). 
Firstly, observe that since $D^*$ is closed and
we have \eqref{eq:Dstar-gap}, its range $\operatorname{Ran}D^*$ is closed. 

As $\operatorname{Ran}D^*$ is closed, by the closed-range theorem,
$\operatorname{Ran}D$ is also closed. Since \eqref{eq:Dstar-gap} gives
$\ker D^*=\{0\}$,
\[
\operatorname{Ran}D=(\ker D^*)^\perp=\mathcal H_{-1},
\qquad
\operatorname{Ran}D^*=(\ker D)^\perp.
\]
Thus, if $v\in\mathcal H_1\cap(\ker D)^\perp$, $v=D^*w$ for some $w\in\mathcal H_{-1}$ and
\[
\|v\|_{L^2}^2
=\langle D^*w,v\rangle_{L^2}
=\langle w,Dv\rangle_{L^2}
\le \|w\|_{L^2}\|Dv\|_{L^2}.
\]
By \eqref{eq:Dstar-gap},
$\|v\|_{L^2}=\|D^*w\|_{L^2}\ge\sqrt2\|w\|_{L^2}$, and therefore
$\|Dv\|_{L^2}^2\ge2\|v\|_{L^2}^2$.
\end{proof}


\begin{lm}[The negative space] \label{lema-index}
Let $\varphi$ be a conformal Hopf map.
Then the negative space of the Jacobi operator of $\varphi$ is exactly
$\ker\left(D|_{\mathcal H_1^\infty}\right)$.
\end{lm}

\begin{proof}
By \eqref{eq:Jacobi-mode},
\[
\mathcal J_\varphi|_{\mathcal{H}_m^\infty}
=
D^*D+m(m-2)\operatorname{Id}.
\]
If $m\le0$ or $m\ge2$, then $m(m-2)\ge0$, so no negative eigenvalue
can occur in these modes. On $\mathcal H_1^\infty$,
\begin{equation}\label{HessH1}
\mathcal J_\varphi=D^*D-\operatorname{Id}.
\end{equation}
For $v\in\ker(D|_{\mathcal H_1^\infty})$ we therefore have
\[
\mathcal J_\varphi v=-v,
\]
so $\ker(D|_{\mathcal H_1^\infty})$ is a negative eigenspace. On the orthogonal
complement of $\ker D$ in $\mathcal H_1^\infty$, the estimate \eqref{DstarD} gives
$D^*D\ge2\operatorname{Id}$, hence
\[
\mathcal J_\varphi\ge\operatorname{Id},
\]
and the whole negative spectral subspace is exactly
$\ker(D|_{\mathcal H_1^\infty})$.
\end{proof}

\begin{lm}[The kernel] \label{lema-nul}
The kernel of the Jacobi operator of a conformal Hopf map is given by
\begin{equation}\label{eq:nullity-two-modes}
\ker \Jj_\varphi
=
\ker\left(D|_{\mathcal H_0^\infty}\right)
\oplus
\ker\left(D|_{\mathcal H_2^\infty}\right).
\end{equation}
\end{lm}

\begin{proof}
Again, by \eqref{eq:Jacobi-mode}, if $m\le-1$ or $m\ge3$ then
\[
m(m-2)>0,
\]
so $\mathcal J_\varphi$ is strictly positive on $\mathcal H_m^\infty$. For
$m=1$, the proof of Lemma~\ref{lema-index} shows that
\[
\mathcal J_\varphi=-\operatorname{Id}
\quad\text{on }\ker D,
\qquad
\mathcal J_\varphi\ge\operatorname{Id}
\quad\text{on }(\ker D)^\perp,
\]
so this weight contributes no Jacobi fields. Finally, for $m=0$ and
$m=2$,
\[
\mathcal J_\varphi=D^*D,
\]
and therefore
\[
\ker(\mathcal J_\varphi|_{\mathcal H_m^\infty})
=
\ker(D|_{\mathcal H_m^\infty}).
\]
Since the weight spaces are mutually orthogonal and invariant under
$\mathcal J_\varphi$, this proves \eqref{eq:nullity-two-modes}.
\end{proof}

The identifications \eqref{eq:assoc-correspondence} and \eqref{eq:D-dbar-correspondence} below seem to be already know in the literature cf. \cite[\S 3]{marin}, but for the convenience of the reader we provide a proof. For the background on complex geometry we refer to \cite{Ball}.

\begin{lm} \label{identification-lm}
Let $\varphi$ be a $d$-conformal Hopf map, set
$E:=\psi^*T^{1,0}\Sigma$ and let $L\simeq\mathcal O(-1)$ be the tautological line bundle over $\Sigma=\CC P^1$. Denote $F_m:=E\otimes L^m$. 

In the smooth category, there is a natural complex-linear one-to-one
correspondence
\begin{equation}\label{eq:assoc-correspondence}
\mathcal H_m^\infty
\simeq
\Gamma(F_m), \qquad V \leftrightarrow s_V.
\end{equation}
Under this correspondence,
\begin{equation}\label{eq:D-dbar-correspondence}
DV=0
\quad\Longleftrightarrow\quad
\bar\partial_{F_m}s_V=0.
\end{equation}
In particular,
\begin{equation*} \label{identification-eq}
\ker\left(D|_{\mathcal H_m^\infty}\right)
\simeq
H^0(\Sigma,F_m),
\end{equation*}
and therefore, when $2d\ge m$,
\begin{equation} \label{identific-dim}
\dim_\CC \ker\left(D|_{\mathcal H_m^\infty}\right)=2d+1-m.
\end{equation}
\end{lm}

\begin{proof}
Recall that the Hopf fibration
is an $\Ss^1$-principal bundle $\Ss^3\xrightarrow{\,\phi\,}\Sigma$ with right action $p\cdot e^{\ii t}=pe^{\ii t}$.  The complex line bundle  $L$ is defined as the associated bundle
\[
L:=\Ss^3\times_{\Ss^1}\CC,
\qquad
(p,z)\sim(pe^{\ii t},e^{-\ii t}z)
\]
and the map
\[
[p,z]\longmapsto zp\in\CC p
\]
identifies $L$ with the tautological bundle  $\mathcal O(-1)$.\\
A section $V$ of $\phi^*E=\varphi^{-1}T^{1,0}\Sigma$ belongs to $\mathcal H_m^\infty$ if it satisfies
\[
V(pe^{\ii t})=e^{-\ii mt}V(p).
\]
For $x=\phi(p)$ define
\begin{equation}\label{eq:sV}
s_V(x):=V(p)\otimes[p,1]^{\otimes m}
\in(E\otimes L^m)_x.
\end{equation}
Replacing by $pe^{\ii t}$, yields
\[
V(pe^{\ii t})=e^{-\ii mt}V(p),
\qquad 
[pe^{\ii t},1]^{\otimes m}
=
e^{\ii mt}[p,1]^{\otimes m},
\]
so \eqref{eq:sV} is independent of the choice of the point of the fibre. Conversely, every
section of $E\otimes L^m$ determines uniquely an equivariant section
of $\phi^*E$. This proves \eqref{eq:assoc-correspondence}.

To translate the equation $DV=0$ in terms of $s_V$, let $\nabla^E=\nabla^\psi$ be the pullback connection of the Levi-Civita connection; as $\psi$ is holomorphic and the
Levi-Civita connection on $T^{1,0}\Sigma$ agrees with the Chern connection,
$\nabla^E$ is the Chern connection of $E$.

 Let $\alpha$ be the
connection form of the Hopf fibration, normalized by
\[
\alpha(\xi)=1,
\qquad
\ker\alpha=\operatorname{span}\{X,Y\}.
\]
The induced unitary connection $\nabla^L$ in the associated line bundle
$L$ coincides with its Chern connection. More precisely, if $\sigma:\mathcal U\to \Ss^3$ is a local section of the Hopf fibration and $\mathcal A:=\sigma^*\alpha$,
then, in the local unitary frame $\ell:=[\sigma,1]$, one has
\[
\nabla^L\ell=\ii \Aa\otimes\ell.
\]
Consequenty $\nabla^{L^m}\ell^m=\ii m\Aa\otimes\ell^m$ and the Chern connection of $F_m=E\otimes L^m$ is 
\[
\nabla^{F_m}
=
\nabla^E\otimes \mathrm{id}_{L^m} + \mathrm{id}_{E}\otimes\nabla^{L^m}.
\]
In particular, writing this locally for $B\in \Gamma(T\Sigma)$ and $s_V=\tilde{V}\otimes\ell^m \in \Gamma(E\otimes L^m)$, where $\tilde{V} =V\circ\sigma$, we get
\[
\nabla_B^{F_m}s_V
=
\bigl(\nabla_B^E \tilde{V} +\ii m \Aa(B)\tilde{V}\bigr)\otimes\ell^m.
\]
On the other hand, let $B^H$ be the horizontal lift of $B$ along $\sigma$:
\[
B^H=d\sigma(B)-\Aa(B)\xi
\]
and observe that we have
\[
\nabla_{B^H}^\varphi V
=
\nabla_B^E \tilde{V}+\ii m\Aa(B)\tilde{V}\, ,
\]
where we have used $\nabla_\xi^\varphi V=-\ii mV$ and $V=\phi^* \tilde{V}$.

Thus the pull-back connection on $\phi^* E \to \Ss^3$ corresponds to the Chern connection on $F_m \to \Sigma$. We use this now to prove \eqref{eq:D-dbar-correspondence}.

At a point of $\mathcal{U}$, the domain of $\sigma$, put
\[
e_1=d\phi(X),\qquad e_2=d\phi(Y)=Je_1.
\]
As $J$ operates on $E$ as multiplication by $\ii$, under the correspondence \eqref{eq:assoc-correspondence},
\[
DV=\nabla_X^\varphi V+J\nabla_Y^\varphi V
\quad\longleftrightarrow\quad
\nabla_{e_1}^{F_m}s_V+\ii\nabla_{e_2}^{F_m}s_V
=
2(\nabla^{F_m})^{0,1}_{e_1}s_V \,.
\]
Since $\nabla^{F_m}$ is the Chern connection,
\[
(\nabla^{F_m})^{0,1}=\bar\partial_{F_m},
\]
and the correspondence ~\eqref{eq:D-dbar-correspondence} follows.

Finally, $T^{1,0}\Sigma\simeq\mathcal O(2)$,
so
\[
E=\psi^*T^{1,0}\Sigma\simeq\mathcal O(2d)
\quad \mbox{ and } \quad
F_m=E\otimes L^m\simeq\mathcal O(2d-m).
\]

By Riemann--Roch on $\Sigma=\mathbb{CP}^1$  (see e.g. \cite[Corollary 5.6.3]{jo}),
\[
\dim_{\mathbb C}H^0(\Sigma,\mathcal O(k))=k+1,
\qquad k\ge0.
\]
Hence, if $2d-m\ge0$,
\[
\dim_\CC\ker(D|_{\mathcal H_m})
=
\dim_\CC H^0(\Sigma,\mathcal O(2d-m))
=
2d+1-m.
\]
This proves the lemma.
\end{proof}

By Lemma \ref{lema-index} and \eqref{identific-dim} with $m=1$,
\[
\dim_\CC\ker(D|_{\mathcal H_1^\infty})=2d.
\]
The negative space is $J$-invariant, hence it is a complex vector space,
whereas the Morse index computes its real dimension. Therefore
\[
\operatorname{index}(\varphi)=2\cdot 2d=4d.
\]
Likewise, Lemma \ref{lema-nul} and \eqref{identific-dim} give
\[
\dim_\CC\ker\mathcal J_\varphi
=
(2d+1)+(2d-1)=4d.
\]
Since the Jacobi kernel is also $J$-invariant,
\[
\operatorname{nullity}(\varphi)=2\cdot4d=8d.
\]
This completes the proof of the main theorem.

\section*{Disclaimer on the use of AI tools}
\noindent The publicly available version of GPT-6 Astra was used in the preparation of this manuscript. All mathematical statements and proofs have been developed and written by the authors, who take
full responsibility for the mathematical content and the final text.

\end{document}